\documentclass[a4paper,francais]{smfart}
\usepackage[T1]{fontenc}
 \usepackage[francais, english]{babel}
 \usepackage{amssymb,url,xspace,smfthm}
\usepackage{smfthm}
 \usepackage{graphicx}
\theoremstyle{plain}

\author[K. DIARRA]{Karamoko DIARRA}
\address{ DER de Math\'ematiques et d'informatique, FAST, Universit\'e des Sciences, des Techniques et des Technologies de Bamako, BP : E $3206$ Mali.}
\email{diarak2005@yahoo.fr}
 \email{karamoko.diarra@univ-rennes1.fr}
\title[Solutions alg\'ebriques]{Solutions alg\'ebriques partielles des \'equations isomonodromiques sur les courbes de genre $2$}

\usepackage[all]{xy}

\begin{document}
\frontmatter

\begin{abstract}
On \'etudie la possibilit\'e de construire des solutions alg\'ebriques partielles des \'equations d'isomonodromie
pour les connections holomorphes de rang $2$ sur les courbes de genre $2$ en adaptant la m\'ethode d'Andreev et Kitaev
par les familles de Hurwitz. Nous classifions tous les cas o\`u la connection est \`a monodromie Zariski dense.
\end{abstract}

\altkeywords{\'Equations diff\'erentielles ordinaires, D\'eformations isomonodromiques, Familles de Hurwitz}
\maketitle
\mainmatter

\section*{Introduction}
L'\'equation de Painlev\'e VI est une \'equation diff\'erentielle ordinaire non-lin\'eaire du second ordre
dont les solutions param\'etrent les d\'eformations isomonodromiques de syst\`emes diff\'erentiels
lin\'eaires de rang $2$ avec $4$ p\^oles simples sur la sph\`ere de Riemann. Les \'equations de Painlev\'e
trouvent leur origine dans les travaux de Paul Painlev\'e : ce sont ``les'' \'equations diff\'erentielles
d'ordre $2$ irr\'eductibles sans singularit\'e mobile. En particulier, leur solution g\'en\'erale est tr\`es
transcendante ; pour autant, il existe des solutions alg\'ebriques ou de de type hyperg\'eom\'etrique.
La recherche des solutions sp\'eciales, notamment alg\'ebriques, a fait l'objet de nombreux travaux 
durant la derni\`ere d\'ecennie \cite{DubrovinMazzocco, Mazzocco1, Mazzocco2, Hitchin1, Hitchin2, 
AndreevKitaev, Kitaev1, Kitaev2,  Boalch1, Boalch2, Boalch3, Boalch4, Boalch5} ;
la classification compl\`ete est donn\'ee dans \cite{LisovyyTykhyy}, ent\'erinant la liste \'etablie dans \cite{Boalch6}. 

Plus g\'en\'eralement, on peut consid\'erer une courbe projective lisse $X$ de genre $g$ sur $\mathbb C$ 
et la donn\'ee de $n$ points distincts (ou du diviseur $D$ r\'eduit correspondant) et consid\'erer les \'equations
d'isomonodromie pour les connections logarithmiques de rang $2$ sur $X$, de diviseur $D$. Lorsque $g=0$,
on obtient les syst\`emes des Garnier (voir \cite{IKSY}), le cas le plus simple $n=4$ correspondant \`a Painlev\'e VI.
Pour les \'equations isomonodromiques en genre $g>0$, voir par exemple \cite{Krichever}. On s'attend \`a ce que la transcendance
de la solution g\'en\'erale croisse avec le genre et le nombre de p\^oles, mais aussi \`a ce qu'il y ait des solutions sp\'eciales,
notamment alg\'ebriques. Les solutions locales sont \`a $N=3g-3+n$ variables (variables de d\'eformation de la courbe et des $n$ p\^oles) ;
on suppose toujours $N>0$. Une premi\`ere m\'ethode pour construire des solutions alg\'ebriques est de consid\'erer
des connexions \`a monodromie finie (voir \cite{Hitchin1,Hitchin2,Boalch1,Boalch2,Boalch5}) ; ceci permet de retrouver
presque toutes les solutions de Painlev\'e VI \`a sym\'etrie d'Okamoto pr\`es (voir \cite{Boalch6}). De telles solutions 
alg\'ebriques existent en tout genre $g$, quel que soit le nombre $n$ de p\^oles. L'autre m\'ethode, que nous avons exploit\'e
dans notre pr\'ec\'edent article \cite{Diarra}, est celle utilis\'ee par Andreev et Kitaev \cite{AndreevKitaev,Kitaev1,Kitaev2}
que nous allons maintent d\'ecrire.

On se donne une connection logarithmique $(E,\nabla)$ sur une courbe $X$ de genre $g$ avec $n$ p\^oles.
On se donne une famille de rev\^etements ramifi\'es $\phi_t:X_t\to X$ d\'ependant alg\'ebriquement d'un param\`etre $t\in T$.
La famille de connexions 
$$(E_t,\nabla_t):=\phi_t^*(E,\nabla)$$
d\'efinie sur la famille de courbes $X_t$ est isomonodromique et d\'efinit une solution alg\'ebrique partielle de l'\'equation 
d'isomonodromie correspondante. Pour un rev\^etement $\phi_t$ g\'en\'eral, ne ramifiant pas au dessus des p\^oles $(E,\nabla)$,
il est facile de voir que l'espace de d\'eformation de $\phi_t$ est de dimension strictement 
plus petite que la dimension de d\'eformation de $(X_t,D_t)$ o\`u $D_t$ est le diviseur des p\^oles de $(E_t,\nabla_t)$.
La d\'eformation isomonodromique param\'etr\'ee par $T$ ne sera que partielle. Pour avoir une solution alg\'ebrique compl\`ete,
il faudra que le rev\^etement ramifie beaucoup au dessus des p\^oles de $(E,\nabla)$ ; en g\'en\'eral, il faudra aussi imposer
aux p\^oles de $(E,\nabla)$ d'avoir une monodromie locale finie de sorte que les relev\'es de ces p\^oles par $\phi_t$ deviennent
des singularit\'es apparentes : on pourra alors les chasser par une transformation birationnelle.

En fait, toute solution alg\'ebrique de Painlev\'e VI peut \^etre reconstruite (\`a sym\'etrie d'Okamoto pr\`es) 
par rev\^etement ramifi\'e d'une connexion $(E,\nabla)$ hyperg\'eom\'etrique ; le cas o\`u $(E,\nabla)$
est \`a monodromie Zariski dense est classifi\'e dans \cite{Doran}.
Dans \cite{Diarra}, nous avons classifi\'e tous les $X$, $(E,\nabla)$  \`a monodromie Zariski dense et $\phi_t:X_t\to X$
donnant lieu \`a une solution alg\'ebrique compl\`ete de l'\'equation d'isomonodromie associ\'ee. On montre que n\'ecessairement
$X$ et $X_t$ sont de genre $0$ avec $(E,\nabla)$ hyperg\'eom\'etrique et on trouve $6$ possibilit\'es de ramification pour $\phi_t$.
On a ainsi obtenu des solutions alg\'ebriques de syst\`emes de Garnier de rang $N=2$ et $3$.

Dans cet article, nous utilisons la m\^eme m\'ethode pour construire des d\'eformations isomonodromiques alg\'ebriques partielles 
de connections holomorphes sur les courbes de genre $g=2$. Nous supposons donc $X_t$ une courbe de genre $2$ et 
$(E_t,\nabla_t)$ n'ayant que des singularit\'es apparentes. On y pensera comme des connexions projectives holomorphes,
c'est \`a dire sans p\^ole. La dimension de d\'eformation de $\phi_t$ est major\'ee par le 
nombre $b$ de points de branchements (valeurs critiques de $\phi_t$) en dehors des p\^oles de $(E,\nabla)$.
Ici, le nombre de variables est $N=3$ et on cherche donc des d\'eformations alg\'ebriques de dimension $b=1$, $2$ ou $3$.
On utilise la m\^eme m\'ethode que dans notre pr\'ec\'edent article \cite{Diarra} : on consid\`ere la structure orbifolde
de $X$ sous-jacente \`a la connexion (voir \cite{Diarra}, section 2) et le fait que le degr\'e du rev\^etement ramifi\'e
agit multiplicativement sur la caract\'eristique d'Euler orbifold.

Dans la section 1, nous montrons que le nombre de param\`etres libres d\'epend de la g\'eom\'etrie de l'orbifolde.
Pour avoir une solution alg\'ebrique compl\`ete, i.e. $b=3$, on doit avoir $\chi>0$ ce qui force $X=\mathbb P^1$ 
et $(E,\nabla)$ \`a provenir d'une \'equation 
hyperg\'eom\'etrique associ\'ee \`a un pavage de la sph\`ere ; en particulier, le groupe de monodromie de $(E,\nabla)$
ou de $(E_t,\nabla_t)$ est fini. R\'eciproquement, toute connexion $(\tilde E,\tilde\nabla)$ \`a monodromie finie 
sur une courbe admet une d\'eformation isomonodromique alg\'ebrique ; une telle d\'eformation est en particulier 
construite par pull-back \`a param\`etre d'une connexion hyperg\'eom\'etrique par le th\'eor\`eme de Klein \cite{Klein}.
Pour $b=2$, on doit avoir $\chi=0$ ce qui force $(E,\nabla)$ et donc $(E_t,\nabla_t)$ \`a \^etre r\'eductible
(voir Proposition \ref{courbeC}).
Ainsi dans le cas o\`u $(E,\nabla)$ est \`a monodromie Zariski dense, on a n\'ecessairement $b\le1$.

Dans les sections 2 et 3, nous dressons la liste des donn\'ees de Hurwitz possibles pour une d\'eformation
\`a $b=1$ param\`etre de rev\^etements lorsque $(E,\nabla)$ est \`a monodromie Zariski dense.
Dans les tables \ref{hypergeom} et \ref{highergenus}, on trouve les degr\'es et ramifications possibles.

Dans la section 4, nous montrons l'existence d'un rev\^etement ramifi\'e pour chaque donn\'ee de Hurwitz 
de notre liste. Pour vraiment classifier les solutions alg\'ebriques correspondantes, il faudrait classifier
tous les rev\^etements possibles modulo conjugaison d'une part (voir \cite{Mednykh}), et modulo l'action naturelle du groupe modulaire 
(ou Mapping Class Group) d'autre part. Mais ceci n\'ecessite d'autres techniques (voir par exemple \cite{Boalch2})
et notamment l'utilisation d'un logiciel de calcul formel ; nous reportons cette \'etude \`a un article ult\'erieur.

Enfin, dans la derni\`ere section, nous construisons explicitement un exemple de degr\'e $6$.
La famille de rev\^etement est alors param\'etr\'ee par la courbe elliptique $E=\{y^2=x^3+1\}$.
\`A tout point $(x_t,y_t)\in E$ de cette courbe, nous associons la courbe de genre $2$
$$X_t:=\{Y^2+(x_t+1)(X^2-1)(3X^4+3(x_t-1)X^2+x_t^2-x_t+1)=0\}$$
avec le rev\^etement ramifi\'e 
$$\phi_t:X_t\to \mathbb P^1\ ;$$
$$\begin{matrix}{(X,Y)\mapsto}{\frac{(x_t-2)^2X^2(Y-3y_t)+4(x_t^2-x_t+1)(y_t-Y)+(x_t^2+2x_t+1-3y_t)X^6-6y_t(x_t-2)X^4}{2(1+x_t)^2X^6}.} \end{matrix}$$
Ce dernier ramifie totalement au dessus de $0$, $1$ et $\infty$ \`a l'ordre $3$, $3$ et $6$ respectivement ;
en dehors de ces $3$ fibres, il a exactement un point de branchement 
$$\phi_t(0,y_t)=\frac{x_t(x_t^3-8)}{4(x_t^3+1)}.$$
Par construction, la courbe $X_t$ est bi-elliptique, rev\^etement double de la courbe elliptique $E$.
Les r\'esultats de cette note sont partiellement issus de notre th\`ese \cite{DiarraThese} effectu\'ee sous la direction de F. Loray \`a l'Universit\'e de Rennes 1.

\section{Structure orbifolde, formule de Riemann-Hurwitz et premi\`ere borne.}

Nous reprenons la d\'emarche de \cite{DiarraThese,Diarra} ; nous renvoyons \`a la section 2 de \cite{Diarra} pour la notion de structure orbifolde.

On se donne un rev\^etement ramifi\'e $\phi : \tilde X\to X$ de degr\'e $d$ entre surfaces de Riemann compactes 
avec $\tilde X$ de genre $2$. On suppose donn\'ee une connexion logarithmique $(E,\nabla)$ de rang $2$ sur $X$ :
$E\to X$ est un fibr\'e vectoriel de rang $2$ et $\nabla:E\to E\otimes\Omega^1_X(D)$ une connexion m\'eromorphe 
de diviseur $D$ r\'eduit, de support disons $\{x_1,\ldots,x_n\}\subset X$. On suppose que $\phi^*(E,\nabla)$ ne poss\`ede
que des singularit\'es apparentes, i.e. devient holomorphe apr\`es transformation de jauge birationnelle.

Associons \`a $(E,\nabla)$ une structure orbifolde \`a singularit\'es coniques sur les $x_i$ : si la monodromie
locale de $(E,\nabla)$ autour de $x_i$ est d'ordre $p_i$, on fixe un point conique d'angle $\frac{2\pi}{p_i}$ en $x_i$.
Notons $p=(p_1,\ldots,p_n)$ la donn\'ee de la structure orbifolde, que l'on peut aussi voir comme une fonction
$p:X\to \mathbb Z_{>0}\cup\{\infty\}$ prenant la valeur $1$ presque partout.
On d\'efinit la caract\'eristique d'Euler de l'orbifolde $(X,p)$ par
$$\chi(X,p):=2-2g+\sum_{i=1}^n(\frac{1}{p_i}-1)=2-2g-n+\frac{1}{p_1}+\cdots+\frac{1}{p_n}$$
o\`u $g$ est le genre de $X$. Lorsque $\chi(X,p)<0$, cette struture orbifolde peut \^etre r\'ealis\'ee
par une m\'etrique de courbure $-1$ sur $X$ (compatible avec la structure conforme de $X$)
et l'aire de $X$ pour cette m\'etrique est alors donn\'ee par $-2\pi\chi(X,p)$. 
La repr\'esentation de monodromie de $(E,\nabla)$ se factorise comme repr\'esentation du groupe
fondamental orbifold de $(X,p)$, et, parmis les structures orbifoldes sur $X$ satisfaisant cette propri\'et\'e,
$(X,p)$ est minimale (d'aire maximale dans le cas hyperbolique). 

On d\'efinit $\tilde p:=\phi^*p$ par :
\begin{itemize}
\item $\tilde p(\tilde x)=p(\phi(\tilde x))$ si $\tilde x\in \tilde X$ est un point r\'egulier
pour $\phi$, 
\item $\tilde p(\tilde x)=p\cdot p(\phi(\tilde x))$ si par contre $\phi$ ramifie \`a l'ordre $p\in\mathbb Z_{>1}$ en $\tilde x$
(i.e. en coordonn\'ees locales $\phi(z)=z^{p}$).
\end{itemize}
Attention, les points coniques deviennent d'angle rationnel en g\'en\'eral : $\tilde p:\tilde X\to \mathbb Q_{>0}\cup\{\infty\}$.
On d\'efinit la caract\'eristique d'Euler de la m\^eme mani\`ere et la formule de Riemann-Hurwitz s'\'ecrit alors :
$$\chi(\tilde X,\tilde p)=d\cdot\chi(X,p).$$ 
Sous notre hypoth\`ese que $\phi^*(E,\nabla)$ ne poss\`ede que des singularit\'es apparentes,
on d\'eduit ais\'ement que $\tilde p$ n'a que des points coniques d'angle multiple de $2\pi$, 
i.e. $\frac{1}{\tilde p}$ est \`a valeurs enti\`eres. Si l'on note 
$$b:=\sum_{\tilde x\in\tilde X}(\frac{1}{\tilde p(\tilde x)}-1)\ \in \mathbb Z_{\ge0},$$
($2\pi b$ est l'exc\'edent d'angles total sur la surface) alors il vient 
$$\chi(\tilde X,\tilde p)=2-2\tilde g+b=b-2.$$
Notons que les points coniques de $\tilde X$ proviennent ou bien du fait qu'on a ``trop'' ramifi\'e
au dessus des $x_i$ (\`a un ordre $p=k p_i$ pour un $k>1$ ce qui contribue pour $k$ dans $b$) 
ou encore du fait que $\phi$ a des points de ramifications ``libres'', c'est \`a dire en dehors des $x_i$.
Ces derniers sont les param\`etres permettant de d\'eformer le rev\^etement ramifi\'e $\phi$ sans 
modifier sa combinatoire au dessus des $x_i$ : pour une telle d\'eformation $\phi_t:\tilde X_t\to X$, les relev\'es 
$\phi_t^*(E,\nabla)$ sont tous \`a singularit\'es apparentes et produisent une d\'eformation isomonodromique
de connexions holomorphes sur $\tilde X_t$. Pour r\'esumer, la formule de Riemann-Hurwitz nous donne :

\begin{prop}\label{courbeC}  Sous les hypoth\`eses et notations pr\'ec\'edentes, on a l'\'egalit\'e
\begin{equation}\label{G}
b=2+\left(2-2g+\sum_{i=1}^n(\frac{1}{p_i} -1)\right)d
\end{equation}
et $b$ majore la dimension de d\'eformation de $\phi$.
\end{prop} 

\subsection{Cas sph\'erique}
Si l'on veut construire une solution alg\'ebrique compl\`ete, alors il faut $b\ge3$, ce qui implique que 
l'orbifolde $(X,p)$ est de caract\'eristique d'Euler $\chi(X,p)>0$. Ces orbifoldes sont classifi\'ees par Klein
dans \cite{Klein} et correspondent aux quotients de $\mathbb P^1$ par les sous-groupes finis de $\mathrm{PGL}_2(\mathbb C)$.
En particulier, le groupe fondamental orbifold est fini, ce qui implique que $(E,\nabla)$ est \`a monodromie finie !
Il en sera de m\^eme pour $\phi_t^*(E,\nabla)$. R\'eciproquement, une connexion holomorphe $(\tilde E,\tilde\nabla)$
sur $\tilde X$ \`a monodromie finie admet une d\'eformation isomonodromique compl\`ete alg\'ebrique.
On pourrait classifier ces d\'eformations de mani\`ere analogue \`a \cite{Boalch2,Boalch5}, c'est \`a dire classifier
les repr\'esentations du groupe fondamental de $\tilde X$ dans les sous-groupes finis de $\mathrm{SL}_2(\mathbb C)$.
Nous laissons cela \`a un article ult\'erieur. Notons que de telles d\'eformations sont automatiquement construites 
par d\'eformation de rev\^etements ramifi\'es au dessus des mod\`eles hyperg\'eom\'etriques standards
$$(p_0,p_1,\infty)=(1,p,p),\ (2,2,p),\ (2,3,3),\ (2,3,4)\ \text{et}\ (2,3,5)$$
(voir \cite{Klein}).

\subsection{Cas euclidien}
Si l'on cherche maintenant une d\'eformation alg\'ebrique de dimension $2$ (codimension $1$)
qui ne soit pas donn\'ee par une monodromie finie, alors on doit avoir $b=2$ et donc $\chi(X,p)=0$.
L'orbifolde est alors un des quotient de la droite complexe $\mathbb C$ par un groupe discret
de covolume fini et on trouve ou bien que $X$ est une courbe elliptique sans point orbifold, ou bien 
$\mathbb P^1$ avec une des structures orbifoldes suivantes :
$$(2,3,6),\ (2,4,4),\ (3,3,3)\ \text{et}\ (2,2,2,2).$$
Notons que $(2,2,\infty)$ est exclu puisque les singularit\'es orbifoldes doivent essentiellement 
dispara\^itre sur le rev\^etement fini $\tilde X$. Dans chacun des cas, le groupe fondamental 
est r\'esoluble (virtuellement ab\'elien) et donc $(E,\nabla)$ doit \^etre r\'eductible ou di\'edral ;
il en sera de m\^eme  pour $\phi_t^*(E,\nabla)$. Ce cas fera l'objet d'un autre article.

\subsection{Cas hyperbolique}
Si $(E,\nabla)$ est \`a monodromie Zariski dense, il ne nous reste que  $b=0$ ou $1$ comme possibilit\'e.
Le cas $b=0$ est rigide, il n'y aura pas de d\'eformation. Notons cependant qu'il a d\'ej\`a \'et\'e 
\'etudi\'e dans \cite{AbuOsmanRosenberger} ; on y trouve la liste des orbifoldes possibles avec
le degr\'e correspondant. Dans la prochaine section, nous donnons la classification similaire
dans le cas $b=1$.

\section{Classification dans le cas o\`u $(X,\nabla)$ est hyperg\'eom\'etrique.}

On suppose dans cette section $X=\mathbb P^1$ et $(E,\nabla)$ avec p\^oles sur $0$, $1$ et $\infty$.
On suppose en outre la monodromie de $(E,\nabla)$ Zariski dense, et donc l'orbifolde sous-jacente $(p_0,p_1,p_\infty)$
hyperbolique : 
$$\frac{1}{p_0}+\frac{1}{p_1}+\frac{1}{p_\infty}<1.$$
On suppose en outre que l'on peut d\'eformer $\phi:\tilde X\to X$, c'est \`a dire $b=1$ avec les notations
pr\'ec\'edentes.

\begin{prop}\label{courbeC0} On a l'\'egalit\'e
\begin{equation}\label{G0}
\left(1-\frac{1}{p_0}-\frac{1}{p_1}-\frac{1}{p_\infty}\right)d=1.
\end{equation}
De plus, on obtient les in\'egalit\'es suivantes : 
$$(1-\frac{1}{p_0}-\frac{1}{p_1})p_{\infty}\le 2\ \ \ \text{et}\ \ \ \frac{1}{3}\le \frac{1}{p_0}+\frac{1}{p_1}<1.$$
\end{prop}

\begin{proof}
Le premier \'enonc\'e se d\'eduit imm\'ediatement de la Proposition \ref{courbeC}.
Le fait que les points coniques disparaissent en haut entra\^ine que le degr\'e $d\ge p_{\infty}$. 
Alors on a $$(1-\sum_{i=0,1,\infty}\frac{1}{p_i})p_{\infty}\le (1-\sum_{i=0,1,\infty}\frac{1}{p_i})d\le 1,$$ 
ce qui nous donne
 \begin{equation}\label{typesfi}
(1-\frac{1}{p_0}-\frac{1}{p_1})p_{\infty}\le 2.
\end{equation}
 Pour obtenir la seconde in\'egalit\'e, il est \'evident que $$\frac{1}{p_0}+\frac{1}{p_1}<1;$$  la minoration de cette in\'egalit\'e provient du fait que $p_{\infty}\ge 3$ (hyperbolicit\'e) : en substituant dans l'in\'egalit\'e (\ref{typesfi}), on obtient $\frac{1}{p_0}+\frac{1}{p_1}\ge \frac{1}{3}$.
 \end{proof}

 Dans le cas hyperbolique, le deuxi\`eme \'enonc\'e de la proposition \ref{courbeC0} nous permet d'obtenir la liste des triplets 
 $(p_0,p_1,p_{\infty})$ suivants :
 \begin{itemize}
\item $(2,3,p_{\infty})$ avec $p_{\infty}=7,\cdots,12$,
 \item $(2,4,p_{\infty})$ avec $p_{\infty}=5,\cdots,8$,
 \item $(2,5,p_{\infty})$ avec $p_{\infty}=5,6$,
 \item $(2,6,6)$,
  \item $(3,3,p_{\infty})$ avec $p_{\infty}=4,5,6$,
\item $(3,4,4)$ et  
\item $(4,4,4)$.
\end{itemize}

Si l'on veut pouvoir d\'eformer $\phi$, on doit avoir un point de branchement libre, 
i.e. en dehors des fibres des points orbifolds $0$, $1$ et $\infty$. Ceci entraine que $\phi$ doit totalement ramifier au dessus de $0$, $1$ et $\infty$
pr\'ecis\'ement \`a l'ordre $p_0$, $p_1$ et $p_{\infty}$ respectivement. En particulier, le degr\'e $d$ du rev\^etement doit \^etre
un multiple des $p_i$. Ceci, compte tenu de  (\ref{G0}) pour $d$ nous donne la liste suivante.

\begin{table}[htdp]
\begin{center}
\begin{tabular}{|c|c|c|}
\hline
$(p_0,p_1,p_\infty)$ & Degr\'e & Ramifications au dessus de $(0\ ;\ 1\ ;\ \infty)$\\
\hline
$(2, 3, 7)$&$42$&$(\underbrace{ 2+\cdots+2 }_{21 fois}\ ;\  \underbrace{ 3+\cdots+3 }_{14 fois}\ ;\ \underbrace{ 7+\cdots+7 }_{6 fois})$\\
\hline
$(2,3,8)$&$24$&$(\underbrace{ 2+\cdots+2 }_{12 fois}\ ;\  \underbrace{ 3+\cdots+3 }_{8 fois}\ ;\ 8+8+8)$\\
\hline
$(2,3,9)$&$18$&$(\underbrace{ 2+\cdots+2 }_{9 fois}\ ;\ \underbrace{ 3+\cdots+3 }_{6 fois}\ ;\ 9+9)$\\
\hline
$(2,3,12)$&$12$&$(\underbrace{ 2+\cdots+2 }_{6 fois}\ ;\ \underbrace{ 3+\cdots+3 }_{4 fois}\ ;\ 12)$\\
\hline
$(2,4,5)$&$20$&$(\underbrace{ 2+\cdots+2 }_{10 fois}\ ;\ \underbrace{ 4+\cdots+4 }_{5 fois}\ ;\ \underbrace{ 5+\cdots+5 }_{4 fois})$\\
\hline
$(2,4,6)$&$12$&$(\underbrace{ 2+\cdots+2 }_{6 fois}\ ;\  4+4+4\ ;\ 6+6)$\\
\hline
$(2,4,8)$&$8$&$(2+2+2+2\ ;\ 4+4\ ;\ 8)$\\
\hline
$(2,5,5)$&$10$&$(\underbrace{ 2+\cdots+2 }_{5 fois}\ ;\ 5+5\ ;\ 5+5)$\\
\hline
$(2,6,6)$&$6$&$(2+2+2\ ;\ 6\ ;\ 6)$\\
\hline
$(3,3,4)$&$12$&$(3+3+3+3\ ;\ 3+3+3+3\ ;\ 4+4+4)$\\
\hline
$(3,3,6)$& $6$& $(3+3\ ;\ 3+3\ ;\ 6)$\\
\hline
$(4,4,4)$&$4$& $( 4\ ;\ 4\ ;\ 4)$\\
\hline
\end{tabular}
\end{center}
\caption{Rev\^etements dans le cas hyperg\'eom\'etrique}
\label{hypergeom}
\end{table}%

Dans chacun des cas,  $\phi^*\nabla$ n'a qu'une singularit\'e apparente libre qui provient du point critique de $\phi$ en dehors de $\{0,1,\infty\}$.
Dans la derni\`ere section, nous expliquerons comment d\'emontrer l'existence de rev\^etements avec ces ramifications
et nous construirons l'avant dernier, de degr\'e $6$, explicitement.

\section{Classification dans le cas o\`u $(X,\nabla)$ n'est pas hyperg\'eom\'etrique.}

Toujours supposant $b=1$, la proposition \ref{courbeC} nous donne
$$\left(2g-2+\sum_{i=1}^n(1-\frac{1}{p_i} )\right)d=1.$$
La caract\'eristique d'Euler de l'orbifolde $(X,p)$ cro\^it rapidement avec le genre $g$
et le nombre $n$ de points conique (voir tables de la page 136 de \cite{Diarra}), 
et on trouve ais\'ement la liste de possibilit\'es exhaustive de la table \ref{highergenus}.
Dans chacun des cas, il y a en outre un point de ramification libre, en dehors des points coniques de $X$.

\begin{table}[htdp]
\begin{center}
\begin{tabular}{|c|c|c|}
\hline
$(g\ ;\ p_1,\ldots,p_n)$ & Degr\'e & Ramifications au dessus de $(p_1,\ldots,p_n)$\\
\hline
$(0\ ;\ 2,2,2,3)$&$6$&$(3+3\ ;\  3+3\ ;\  3+3\ ;\  2+2+2)$\\
\hline
$(0\ ;\ 2,2,2,4)$&$4$&$(2+2\ ;\ 2+2\ ;\ 2+2\ ;\ 4)$\\
\hline
$(0\ ;\ 2,2,2,2,2)$&$2$&$(2\ ;\ 2\ ;\ 2\ ;\ 2\ ;\ 2)$\\
\hline
$(1\ ;\ 2)$&$2$&$(2)$\\
\hline
\end{tabular}
\end{center}
\caption{Rev\^etements dans le cas non hyperg\'eom\'etrique}
\label{highergenus}
\end{table}

\section{Existence de rev\^etements}

Jusqu'\`a maintenant, nous n'avons consid\'er\'e que les obstructions donn\'ees par la formule de Riemann-Hurwitz,
ce qui nous a conduit aux tables \ref{hypergeom} et \ref{highergenus}. Il reste \`a montrer l'existence de tels rev\^etements,
c'est le {\it probl\`eme de Hurwitz}. Il suffit pour cela de trouver, par exemple dans le cas hyperg\'eom\'etrique, 
pour chaque donn\'ee $d$ et $(p_0,p_1,p_\infty)$, une repr\'esentation 
$$\pi_1(\mathbb P^1\setminus\{0,1,\infty,t\})\to\mathrm{Perm}(1,\ldots,d)$$
du groupe fondamental dans le groupe des permutations satisfaisant les propri\'et\'es suivantes.
Si l'on note $\sigma_0,\sigma_1,\sigma_\infty,\sigma_t\in\mathrm{Perm}(1,\ldots,d)$ les images des g\'en\'erateurs standards
du groupe fondamental, alors on veut 
\begin{itemize}
\item $\sigma_i$ est conjugu\'e au produit de $p_i$ permutations cycliques disjointes de longueur $\frac{d}{p_i}$ pour $i=0,1,\infty$,
\item $\sigma_t$ est une transposition,
\item $\sigma_0\circ\sigma_1\circ\sigma_\infty\circ\sigma_t=\text{identit\'e}$,
\item le groupe engendr\'e $\langle \sigma_0,\sigma_1,\sigma_\infty\rangle$ est transitif.
\end{itemize}
Par exemple, le rev\^etement de degr\'e $4$ en bas de la table \ref{hypergeom} est r\'ealis\'e par
$$\sigma_0=(1234),\ \sigma_1=(1324),\ \sigma_\infty=(1342)\ \ \ \text{et}\ \ \ \sigma_t=(12).$$
Nous d\'etaillons maintenant les deux rev\^etements de degr\'e $6$ de la table \ref{hypergeom}.

Pour $(p_0,p_1,p_\infty)=(2,6,6)$, remarquons qu'il existe un rev\^etement de degr\'e $6$ ramifiant 
uniquement au dessus de $0$, $1$ et $\infty$ et d\'esingularisant l'orbifolde $(\tilde p_0,\tilde p_1,\tilde p_\infty)=(2,6,3)$: c'est le quotient
de la courbe elliptique qui poss\`ede un automorphisme d'ordre $6$ par le groupe engendr\'e.
Il existe donc $\tilde\sigma_0$, $\tilde\sigma_1$ et $\tilde\sigma_\infty$ satisfaisant
\begin{itemize}
\item $\tilde \sigma_i$ est conjugu\'e au produit de $\tilde p_i$ permutations cycliques disjointes de longueur $\frac{d}{\tilde p_i}$ pour $i=0,1,\infty$,
\item $\tilde \sigma_0\circ\tilde \sigma_1\circ\tilde \sigma_\infty=\text{identit\'e}$,
\item le groupe engendr\'e $\langle \tilde \sigma_0,\tilde \sigma_1,\tilde \sigma_\infty\rangle$ est transitif.
\end{itemize}
En particulier, \`a conjugaison pr\`es, on peut supposer 
$\tilde\sigma_\infty=(123)(456)$.
Il suffit alors de poser $\sigma_i:=\tilde\sigma_i$ pour $i=0,1$, $\sigma\infty=(123456)$ et $\sigma_t=(14)$
et on a les relations attendues pour $(2,6,6)$.
De la m\^eme mani\`ere, on ram\`ene les cas $(2,4,8)$ et $(2,3,12)$ aux cas $(2,4,4)$ et $(2,3,6)$ respectivement.
Ces derniers sont d\'esingularis\'es par des rev\^etements elliptiques de degr\'es $4$ et $6$ respectivement,
mais on peut les composer par une isog\'enie de degr\'e $2$ ce qui nous donne les bons degr\'es $8$ et $12$.

Pour $(p_0,p_1,p_\infty)=(3,3,6)$, remarquons qu'il existe un rev\^etement de degr\'e $6$ ramifiant 
uniquement au dessus de $0$, $1$ et $\infty$ et d\'esingularisant l'orbifolde $(\tilde p_0,\tilde p_1,\tilde p_\infty)=(3,6,6)$: 
voir \cite{AbuOsmanRosenberger} (section 3.7). En notant comme avant $\tilde\sigma_0$, $\tilde\sigma_1$ et $\tilde\sigma_\infty$
les permutations d\'ecrivant ce rev\^etement, on peut supposer $\tilde\sigma_\infty=(123456)$.
Il suffit alors de poser $\sigma_i:=\tilde\sigma_i$ pour $i=0,1$, $\sigma\infty=(123)(456)$ et $\sigma_t=(14)$
et on a les relations attendues pour $(3,6,3)$. De la m\^eme mani\`ere, on ram\`ene les cas $(2,5,5)$, $(2,4,6)$ et $(2,3,9)$ 
respectivement aux cas $(2,5,10)$, $(2,4,12)$ et $(2,3,18)$ de \cite{AbuOsmanRosenberger}.

On voit facilement que l'orbifolde hyperg\'eom\'etrique $(3,3,4)$ est rev\^etement double de $(2,3,8)$ 
(ramifiant sur les points d'ordre $2$ et $8$) et que $(4,4,4)$ est rev\^etement triple de $(3,3,4)$
(ramifiant sur les deux points d'ordre $3$). En composant ces rev\^etements avec celui d'ordre $4$ 
construit au dessus de $(4,4,4)$ au tout d\'ebut, on obtient les rev\^etements recherch\'es pour 
$(2,3,8)$ et $(3,3,4)$. De la m\^eme mani\`ere, l'orbifolde hyperg\'eom\'etrique $(2,5,5)$ 
s'obtient par rev\^etement double de $(2,4,5)$, ce qui nous donne l'existence du rev\^etement 
recherch\'e pour $(2,4,5)$.

Pour le rev\^etement $(2,3,7)$ de degr\'e $42$, on peut le d\'ecomposer en un rev\^etement
de degr\'e $7$ par l'orbifolde de genre $0$ et de structure orbifolde $(2,2,2,3)$ (voir \cite{PascaliPetronio}, en haut de la table 2)
compos\'e avec un rev\^etement de degr\'e $6$ comme en haut de notre table \ref{highergenus}.
Enfin, pour cette derni\`ere table, l'existence des rev\^etements de degr\'e $6$ et $4$ se ram\`ene ais\'ement, par les m\'ethodes 
perturbatives d\'ecrites au dessus, aux cas de rev\^etements galoisiens (sans branchement libre) des 
orbifoldes de genre $0$ et de structure orbifolde $(2,2,3,3)$ et $(2,2,4,4)$ (voir \cite{HKNR}, list 1).

\section{Une famille de rev\^etements de degr\'e $6$}\label{RevDeg6genre2}

On veut construire un exemple explicite de rev\^etements $\phi:C_2\to \mathbb P^1$ de degr\'e $6$ 
ramifiant totalement au dessus de $0$, $1$ et $\infty$ \`a l'ordre $3$, $3$ et $6$ respectivement
et ayant un seul autre point de ramification simple, que nous pourrons d\'eformer.
Nous allons le construire comme composition de deux rev\^etements galoisiens $\phi=\phi_2\circ\phi_1$.
Le premier $\phi_1:E\to \mathbb P^1$ est de degr\'e $3$, ramifiant totalement
\`a l'ordre $3$ au dessus de $0$, $1$ et $\infty$ et nulle-part ailleurs : $E$ est 
la courbe elliptique poss\'edant une sym\'etrie d'ordre $3$ et $\phi_1$ est le quotient.
Le second $\phi_2:C_2\to E$ est un rev\^etement bielliptique de degr\'e $2$ ramifiant 
sur $\phi_1^{-1}(\infty)$ ainsi que sur un autre point libre, que nous pourrons faire varier
sur ma courbe elliptique $E$.

Le premier rev\^etement est par exemple donn\'e par
$$\phi_1:E=\{y^2=x^3+1\}\to\mathbb P^1_z\ ;\ (x,y)\mapsto z=\frac{y+1}{2}.$$
Pour calculer le second, on choisit un point $(x_t,y_t)\in E$, puis on compose le rev\^etement bielliptique
ramifiant au dessus des points $(x_t,\pm y_t)$ avec la translation par $(x_t,y_t)$. De cette mani\`ere,
le point $(x_t,-y_t)$ sera envoy\'e \`a l'infini de $E$, puis sur le point $\infty\in\mathbb P^1_z$ par $\phi_1$,
alors que le point $(x_t,y_t)$ sera envoy\'e sur son double dans $E$, puis par $\phi_1$, sur le point 
de branchement libre de $\mathbb P^1_z$. On notera $(x_1,y_1)$ les coordonn\'ees non translat\'ees
et $(x,y)$ les coordonn\'ees translat\'ees de sorte que $(x,y)=(x_1,y_1)\oplus(x_t,y_t)$. Rappelons 
les formules de translation sur $E$ :
$$\left\{\begin{matrix}
 x&=&a^2-x_1-x_t\\ y&=&-(ax+b)
 \end{matrix}\right.\ \ \ \text{o\`u}\ \ \ 
\left\{\begin{matrix}
 a&=&\frac{y_1-y_t}{x_1-x_t}\\ b&=&\frac{x_1y_t-x_ty_1}{x_1-x_t}
 \end{matrix}\right.$$
\`A partir de maintenant, notons plut\^ot $\phi_1:E\to\mathbb P^1$ le rev\^etement de degr\'e $3$
vu depuis les cooordonn\'ees non translat\'ees $(x_1,y_1)$. En simplifiant les puissances de $y_1$
et $y_t$ avec l'\'equation de la courbe elliptique $E$, on trouve
$$\phi_1(x_1,y_1)=\frac{-(3x_t^2x_1+x_t^3+4)y_1+(1+y_t)x_1^3+3x_t(y_t-1)x_1^2+3x_t^2x_1+4y_t-x_t^3}{2(x_1-x_t)^3}.$$
Les $3$ points de ramification sont donn\'es par
$$\left\{\begin{matrix}
E&\stackrel{\oplus(x_t,y_t)}{\longrightarrow}&E&\stackrel{\frac{y+1}{2}}{\longrightarrow}&\mathbb P^1\\
(x_1,y_1)&\mapsto&(x,y)&\mapsto&z\\
\left(\frac{2(1-y_t)}{x_t^2},\frac{4y_t-x_t^3-4}{x_t^3}\right)&\mapsto&(0,-1)&\mapsto&0\\
\left(\frac{2(1+y_t)}{x_t^2},\frac{4y_t+x_t^3+4}{x_t^3}\right)&\mapsto&(0,1)&\mapsto&1\\
\left(x_t,-y_t\right)&\mapsto&(\infty,\infty)&\mapsto&\infty
 \end{matrix}\right.$$
Le rev\^etement bielliptique est obtenu par exemple en introduisant la variable $X^2=\frac{x_1-x_t}{x_1+1}$
(notons que $x_1$ ramifie d\'ej\`a sur $-1$) et la variable $Y=y_1(X^2-1)^2$ afin de mettre l'\'equation
de $C_2$ sous forme hyperelliptique :
$$C_2=\{Y^2+(x_t+1)(X^2-1)(3X^4+3(x_t-1)X^2+x_t^2-x_t+1)\}$$
et la projection sur $E$ est donn\'ee par :
$$\phi_2:C_2\to E\ ;\ (X,Y)\mapsto (x_1,y_1)=\left(-\frac{X^2+x_t}{X^2-1},\frac{Y}{(X^2-1)^2}\right).$$
Les deux points de ramification sont donn\'es par 
$$\left\{\begin{matrix}
C_2&\stackrel{\phi_2}{\longrightarrow}&E\\
(X,Y)&\mapsto&(x_1,y_1)\\
(0,y_t)&\mapsto&(x_t,y_t)\\
(0,-y_t)&\mapsto&(x_t,-y_t)
 \end{matrix}\right.$$
 En composant $\phi_1$ et $\phi_2$ (et en simplifiant les puissances de $y_t$), 
 il vient
$$\begin{matrix}{\phi(X,Y)=}{\frac{(x_t-2)^2X^2(Y-3y_t)+4(x_t^2-x_t+1)(y_t-Y)+(x_t^2+2x_t+1-3y_t)X^6-6y_t(x_t-2)X^4}{2(1+x_t)^2X^6}.} \end{matrix}$$


{\small
\begin{thebibliography}{99}

\bibitem{AbuOsmanRosenberger} M. T. Abu Osman et G. Rosenberger, Embedding property of surface groups. 
Bull. Malaysian Math. Soc. 3 (1980) 21-27.

\bibitem{AndreevKitaev} F. V. Andreev et A. V. Kitaev, Transformations $RS^2_4(3)$ of the ranks $\le 4$ and algebraic solutions of the sixth Painlev\'e equation. Comm. Math. Phys. 228 (2002) 151-176.

\bibitem{Boalch1} P.  Boalch, From Klein to Painlev\'e via Fourier, Laplace and Jimbo. Proc. London Math. Soc. 90 (2005) 167-208. 

\bibitem{Boalch2} P.  Boalch,  The fifty-two icosahedral solutions to Painlev\'e VI. J. Reine Angew. Math. 596 (2006) 183-214. 

\bibitem{Boalch3} P. Boalch,  Six results on Painlev\'e VI. Th\'eories asymptotiques et \'equations de Painlev\'e, 1-20, S\'emin. Congr., 14, Soc. Math. France, Paris, 2006. 

\bibitem{Boalch4} P.  Boalch, Higher genus icosahedral Painlev\'e curves. Funkcial. Ekvac. 50 (2007), no. 1, 19-32. 

\bibitem{Boalch5} P. Boalch,  Some explicit solutions to the Riemann-Hilbert problem. Differential equations and quantum groups, 85-112, IRMA Lect. Math. Theor. Phys., 9, Eur. Math. Soc., Z\"urich, 2007. 

\bibitem{Boalch6} P.  Boalch, Towards a non-linear Schwarz's list. The many facets of geometry, 210-236, Oxford Univ. Press, Oxford, 2010.

\bibitem{Chiarellotto} B. Chiarellotto, On Lam\'e Operators which are Pullbacks of Hypergeometric Ones. 
Trans. Amer. Math. Soc. 347 (1995) 2735-2780.

\bibitem{DiarraThese} K. Diarra, Construction de d\'eformations isomonodromiques par rev\^etements, 
Doctoral thesis, Universit\'e de Rennes 1 (2011). http://tel.archives-ouvertes.fr/

\bibitem{Diarra} K. Diarra, Construction et classification de certaines solutions alg\'ebriques des syst\`emes de Garnier. 
Bull. Braz. Math. Soc., New Series 44  (2013) 1-26.

\bibitem{Doran} C. F. Doran, Algebraic and Geometric Isomonodromic Deformations. J. Differential Geometry 59 (2001) 33-85.

\bibitem{DubrovinMazzocco} B. Dubrovin et M.  Mazzocco, 
Monodromy of certain Painlev\'e-VI transcendents and reflection groups. 
Invent. Math. 141 (2000) 55-147. 

\bibitem{Hitchin1} N. Hitchin, 
Poncelet polygons and the Painlev\'e equations. Geometry and analysis (Bombay, 1992), 151-185, Tata Inst. Fund. Res., Bombay, 1995. 

\bibitem{Hitchin2} N. Hitchin, A lecture on the octahedron. 
Bull. London Math. Soc. 35 (2003) 577-600. 

\bibitem{HKNR} A. Hulpke, T. Kuusalo, M. N\"a\"at\"anen et G. Rosenberger, 
On orbifold coverings by genus 2 surfaces.
Sci. Ser. A Math. Sci. (N.S.) 11 (2005) 45-55. 

\bibitem{IKSY} K. Iwasaki, H. Kimura, S. Shimomura et M. Yoshida, 
From Gauss to Painlev\'e. 
A modern theory of special functions. Aspects of Mathematics, E16. Friedr. Vieweg and Sohn, Braunschweig, 1991.

\bibitem{Kitaev1} A. V. Kitaev, Special functions of isomonodromy type, rational transformations of the spectral parameter, and algebraic solutions of the sixth Painlev\'e equation.  Algebra i Analiz 14 (2002) 121-139.

\bibitem{Kitaev2} A. V. Kitaev,  Grothendieck's dessins d'enfants, their deformations, and algebraic solutions of the sixth Painlev\'e and Gauss hypergeometric equations. Algebra i Analiz 17 (2005) 224-275.


\bibitem{Klein} F. Klein, Vorlesungen \"uber das Ikosaedar, B. G. Teubner, Leipzig (1884).

\bibitem{Krichever} I. Krichever, 
Isomonodromy equations on algebraic curves, canonical transformations and Whitham equations. 
Mosc. Math. J. 2 (2002) 717-752. 

\bibitem{LisovyyTykhyy} O. Lisovyy et Y. Tykhyy, Algebraic solutions of the sixth Painlev\'e equation.
	arXiv:0809.4873 [math.CA]

\bibitem{Mazzocco1} M. Mazzocco, Rational solutions of the Painlev\'e VI equation. 
Kowalevski Workshop on Mathematical Methods of Regular Dynamics (Leeds, 2000). 
J. Phys. A 34 (2001) 2281-2294. 

\bibitem{Mazzocco2} M. Mazzocco, 
Picard and Chazy solutions to the Painlev\'e VI equation. 
Math. Ann. 321 (2001) 157-195. 

\bibitem{Mednykh} A. Mednykh, 
Counting conjugacy classes of subgroups in a finitely generated group. 
J. Algebra 320 (2008) 2209-2217. 

\bibitem{PascaliPetronio} M. A. Pascali et C. Petronio, 
Branched covers of the sphere and the prime-degree conjecture. 
Ann. Mat. Pura Appl. 191 (2012) 563-594. 

\bibitem{PervovaPetronio1} E. Pervova et C. Petronio,
On the existence of branched coverings between surfaces with prescribed branch data. I.
Algebr. Geom. Topol. 6 (2006) 1957-1985.

\bibitem{PervovaPetronio2} E. Pervova et C. Petronio,
On the existence of branched coverings between surfaces with prescribed branch data. II.  
J. Knot Theory Ramifications 17 (2008) 787-816. 

\end{thebibliography}
}
\end{document}